\documentclass[review]{elsarticle}
\usepackage{amsmath}
\usepackage{graphicx,setspace}
\usepackage{amsfonts,amsmath, amssymb, amsthm}
\usepackage{mathrsfs}
\usepackage{epstopdf}
\usepackage{float}
\usepackage{lineno,hyperref}
\usepackage{color}
\usepackage{subfigure}
\usepackage{bm}
\usepackage{graphicx}
\usepackage{amssymb, amsthm}
\usepackage{mathrsfs}
\usepackage{epstopdf}
\usepackage{float}
\usepackage{lineno,hyperref}
\usepackage{color}
\usepackage{subfigure}
\usepackage{cases}

\modulolinenumbers[5]
\numberwithin{equation}{section}
\journal{}

\begin{document}

\newtheorem{definition}{Definition}
\newtheorem{lemma}{Lemma}
\newtheorem{remark}{Remark}
\newtheorem{theorem}{Theorem}
\newtheorem{proposition}{Proposition}
\newtheorem{assumption}{Assumption}
\newtheorem{example}{Example}
\newtheorem{corollary}{Corollary}
\def\e{\varepsilon}
\def\Rn{\mathbb{R}^{n}}
\def\Rm{\mathbb{R}^{m}}
\def\E{\mathbb{E}}
\def\hte{\bar\theta}
\def\cC{{\mathcal C}}
\numberwithin{equation}{section}

\begin{frontmatter}

\title{Effective Approximation for a Stochastic System with Large Potential}

\author{\bf\normalsize{
Ao Zhang\footnote{School of Mathematics and Statistics, \& Center for Mathematical Sciences,  Huazhong University of Sciences and Technology, Wuhan 430074,  China. Email: \texttt{zhangao1993@hust.edu.cn}},
Jinqiao Duan\footnote{Department of Applied Mathematics, Illinois Institute of Technology, Chicago, IL 60616, USA. Email: \texttt{duan@iit.edu}}
}}

\begin{abstract}
This letter is about effective approximation for a stochastic parabolic equation with a large potential in a periodic medium. Under a condition on the spectral properties of the associated cell problem, we prove that the solution can be approximately factorized as the product of a fast oscillating cell eigenfunction and of a solution of an effective equation.
\end{abstract}


\end{frontmatter}

\section{Introduction}

There is a vast literature on periodic and quasi-periodic homogenization of partial differential equations. Contrasted with deterministic homogenization, very few results are available for homogenization of stochastic partial differential equations \cite{AJ20, JW14, Wa07, WD07, BM16}. In some circumstances, randomness has to be taken into account and it often occurs through a random potential. In this letter, we consider the effective approximation for a stochastic parabolic equation with periodic coefficients, by Bloch wave theory and two-scale convergence. The Bloch transform is a generalization of Fourier transform that leaves invariant periodic functions \cite{Al04, Al05, Co97}. The method of two-scale convergence is a powerful tool for studying homogenization problems for partial differential equations with periodically oscillating coefficients \cite{Al92, Ngu98}. The theory of the two-scale convergence from the periodic to the stochastic setting was considered by Bourgeat, Mikeli\'c and Wright \cite{BMW94}, using techniques from ergodic theory.

We study the following parabolic equation with multiplicative noise
\begin{equation}\label{Equ.1}
\begin{cases}
\frac{\partial u_{\varepsilon}}{\partial t} - {\rm div}(\sigma(\frac{x}{\varepsilon})\nabla u_{\varepsilon})+(\varepsilon^{-2}c(\frac{x}{\varepsilon})+d(x,\frac{x}{\varepsilon}))u_{\varepsilon}+g(\frac{x}{\varepsilon})u_{\varepsilon}\frac{dW(t)}{dt}=0  \quad \text{in} \ D\times [0,T], \\
u_\varepsilon=0 \quad \text{on} \ \partial D\times [0,T], \\
u_\varepsilon(0,x)=u^0_{\varepsilon}(x)  \quad \text{in} \ D,
\end{cases}
\end{equation}
where $D\subset\mathbb{R}^N$ is an open domain, $0<T<\infty$, and the unknown $u_{\varepsilon}(t, x)$ is a function from $[0, T]\times D$ into $\mathbb{C}$. The matrix $\sigma(y)$ is real and bounded function defined for $y\in {\mathbb{T}^N}$ (the unit torus), and the coefficients $c(y)$, $d(x, y)$ and $g(y)$ are real and bounded functions defined for $x\in D$ and $y\in {\mathbb{T}^N}$. Furthermore, the real-valued Wiener processes $W(t)$ is defined on the complete probability space $(\Omega, \mathcal{F}, \mathbb{P})$ endowed with the canonical filtration $(\mathcal{F}_t)_{t\in[0, T]}$.

We are interested in the behavior of the solution $u_{\varepsilon}(t, x, \omega)$ as scale parameter $\varepsilon\to 0$. We introduce a cell problem,
\begin{equation}\nonumber
-(\text{div}_y+2i\pi\theta)\left(\sigma(y)(\nabla_y+2i\pi\theta)\psi_n\right)+c(y)\psi_n=\lambda_n(\theta)\psi_n \quad \text{in}\ \mathbb{T}^N,
\end{equation}
where $\theta\in\mathbb{T}^N$ is a parameter and $(\lambda_n(\theta), \psi_n(y, \theta))$ is the $n^{th}$ eigenpair. Under some assumptions, we focus on higher energy initial data and consider well-prepared initial data of the type
\begin{equation}
u_{\varepsilon}^0(x)=\psi_n(\frac{x}{\varepsilon},\theta^n)e^{2i\pi\frac{\theta^n\cdot x}{\varepsilon}}v^0(x).
\end{equation}
We shall prove in Theorem 1 that the solution of (\ref{Equ.1}) satisfies
\begin{equation}
u_{\varepsilon}(t, x, \omega)\approx e^{i\frac{\lambda_n(\theta^n)t}{\varepsilon^2}}e^{2i\pi\frac{\theta^n\cdot x}{\varepsilon}}\psi_n(\frac{x}{\varepsilon},\theta^n)v(t, x, \omega),
\end{equation}
where $v(t, x, \omega)$ is the unique solution of an effective homogenized stochastic parabolic equation.

The letter is organized as follows. In Section 2, we define the functional spaces, make some assumptions, and introduce some results on Bloch theory and two-scale convergence. In Section 3, we state and prove our effective approximation result.

\section{Preliminaries}
Recall that $L^2(D)$ denotes the complex $L^2$ space with usual scalar product $(\cdot,\cdot)$ and norm $\|\cdot\|$ , and $H^1(D)$ denotes the Sobolev space $H^1$.

We make the following assumptions on this stochastic parabolic equation (\ref{Equ.1}).
\par
{\bf Hypothesis  H.1. }The coefficients $\sigma(y)$, $c(y)$ and $g(y)$ are real measurable bounded periodic functions, i.e. their entries belong to $L^{\infty}(\mathbb{T}^N)$, while $d(x, y)$ is real measurable and bounded with respect to $x$, and periodic continuous with respect to $y$, i.e. its entries belong to $L^{\infty}(D; C(\mathbb{T}^N))$.

\par
{\bf Hypothesis  H.2. }The matrix $ \sigma $ is symmetric and uniformly coercive, i.e. there exists $\nu > 0$ such that for a.e. $y\in \mathbb{T}^N$,
\begin{equation}\nonumber
\sigma(y)\xi\cdot\xi\geqslant\nu|\xi|^2, \quad \text{for all}\ \xi\in \mathbb{R}^N.
\end{equation}
\subsection{Bloch Spectrum}
We recall the so-called Bloch spectral cell equation
\begin{equation}\label{Cell}
-(\text{div}_y+2i\pi\theta)\left(\sigma(y)(\nabla_y+2i\pi\theta)\psi_n\right)+c(y)\psi_n=\lambda_n(\theta)\psi_n \quad \text{in}\ \mathbb{T}^N,
\end{equation}
which, as a compact self-adjoint complex-valued operator on $L^2(\mathbb{T}^N)$, admits a countable sequence of real increasing eigenvalues $(\lambda_n)_{n\geqslant 1}$ and normalized eigenfunctions $(\psi_n)_{n\geqslant 1}$ with $\|\psi_n\|_{L^2(\mathbb{T}^N)}=1$. The dual parameter $\theta$ is called the Bloch frequency and it runs in the dual cell of $\mathbb{T}^N$, i.e. by periodicity it is enough to consider $\theta\in\mathbb{T}^N$.

In the sequel, we shall consider an energy level $n\geq 1$ and a Bloch parameter $\theta^n\in\mathbb{T}^N$ such that the eigenvalue $\lambda_n(\theta^n)$ satisfies some assumptions. Depending on these precise assumptions we obtain different homogenized limits for the stochastic parabolic equation (\ref{Equ.1}).
\par
{\bf Hypothesis  H.3. }
$\lambda_n(\theta^n)$  is a simple eigenvalue; $\theta^n$ is a critical point of $\lambda_n(\theta)$, i.e. $\nabla_{\theta}\lambda_n(\theta^n)=0$.

Under Hypothesis  H.3., it is a classical matter to prove that the $n^{th}$ eigenpair of (\ref{Cell}) is smooth in a neighborhood of $\theta^n$\cite{Ka66}. Introducing the operator $\mathbb{A}_n(\theta)$ defined on $L^2(\mathbb{T}^N)$ by
\begin{equation}\label{Oper.1}
\mathbb{A}_n(\theta)\psi=-(\text{div}_y+2i\pi\theta)\left( \sigma(y)(\nabla_y+2i\pi\theta)\psi\right)+c(y)\psi-\lambda_n(\theta)\psi.
\end{equation}
Denoting by $(e_k)_{1\leqslant k\leqslant N}$ the standard basis of $\mathbb{R}^N$ and by
$(\theta_k)_{1\leqslant k\leqslant N}$ the components of $\theta$, the first derivative of (\ref{Cell}) satisfies
\begin{equation}\label{Oper.2}
\mathbb{A}_n(\theta)\frac{\partial\psi_n}{\partial\theta_k}=2i\pi e_k\sigma(y)(\nabla_y+2i\pi\theta)\psi_n+(\text{div}_y+2i\pi\theta)(\sigma(y)2i\pi e_k\psi_n)+\frac{\partial\lambda_n}{\partial\theta_k}(\theta)\psi_n,
\end{equation}
and the second derivative of (\ref{Cell}) satisfies
\begin{equation}\label{Oper.3}
\begin{split}
\mathbb{A}_n(\theta)\frac{\partial^2\psi_n}{\partial\theta_k\partial\theta_l}=&2i\pi e_k\sigma(y)(\nabla_y+2i\pi\theta)\frac{\partial\psi_n}{\partial\theta_l}+(\text{div}_y+2i\pi\theta)\left(\sigma(y)2i\pi e_k\frac{\partial\psi_n}{\partial\theta_l}\right)\\
&+2i\pi e_l\sigma(y)(\nabla_y+2i\pi\theta)\frac{\partial\psi_n}{\partial\theta_k}+(\text{div}_y+2i\pi\theta)\left(\sigma(y)2i\pi e_l\frac{\partial\psi_n}{\partial\theta_k}\right)\\
&+\frac{\partial\lambda_n}{\partial\theta_k}(\theta)\frac{\partial\psi_n}{\partial\theta_l}+\frac{\partial\lambda_n}{\partial\theta_l}(\theta)\frac{\partial\psi_n}{\partial\theta_k}\\
&-4\pi^2e_k\sigma(y)e_l\psi_n-4\pi^2e_l\sigma(y)e_k\psi_n+\frac{\partial^2\lambda_n}{\partial\theta_k\partial\theta_l}(\theta)\psi_n.
\end{split}
\end{equation}
Under Hypothesis  H.3., the equations (\ref{Oper.2}) and (\ref{Oper.3}) simplify for $\theta=\theta^n$ and we find
\begin{equation}
\frac{\partial\psi_n}{\partial\theta_k}=2i\pi\zeta_k, \quad  \frac{\partial^2\psi_n}{\partial\theta_k\partial\theta_l}=-4\pi^2\chi_{kl},
\end{equation}
where $\zeta_k$ is the solution of
\begin{equation}\label{SolutionZeta}
\mathbb{A}_n(\theta^n)\zeta_k=e_k\sigma(y)(\nabla_y+2i\pi\theta^n)\psi_n+(\text{div}_y+2i\pi\theta^n)(\sigma(y)e_k\psi_n) \quad\text{in}\ \mathbb{T}^N,
\end{equation}
and $\chi_{kl}$ is the solution of
\begin{equation}\label{SolutionChi}
\begin{split}
\mathbb{A}_n(\theta^n)\chi_{kl}=&e_k\sigma(y)(\nabla_y+2i\pi\theta^n)\zeta_l+(\text{div}_y+2i\pi\theta^n)(\sigma(y)e_k\zeta_l)\\
&+e_l\sigma(y)(\nabla_y+2i\pi\theta^n)\zeta_k+(\text{div}_y+2i\pi\theta^n)(\sigma(y)e_l\zeta_k)\\
&+e_k\sigma(y)e_l\psi_n+e_l\sigma(y)e_k\psi_n-\frac{1}{4\pi^2}\frac{\partial^2\lambda_n}{\partial\theta_k\partial\theta_l}(\theta^n)\psi_n \quad\text{in}\ \mathbb{T}^N.
\end{split}
\end{equation}
There exists a unique solution of (\ref{SolutionZeta}), up to the addition of a multiple of $\psi_n$. Indeed, the right hand side of (\ref{SolutionZeta}) satisfies the required compatibility condition (i.e. it is orthogonal to $\psi_n$) because $\zeta_k$ is just a multiple of the partial derivative of $\psi_n$ with respect to $\theta$ which necessarily exists. Similarly, there exists a unique solution of (\ref{SolutionChi}), up to the addition of a multiple of $\psi_n$. The compatibility condition of (\ref{SolutionChi}) yields a formula for the Hessian matrix $\nabla_{\theta}\nabla_{\theta}\lambda_n(\theta^n)$, see \cite{Al04, Al05}.

\subsection{Two-Scale Convergence}
We now summarize several results about the two-scale convergence that we will need. For the results stated without proofs, see \cite{Al92, BMW94, BM16}. We denote by $C_{\#}(\mathbb{T}^N)$ the space of functions from $C(\bar{\mathbb{T}}^N)$ that have $\mathbb{T}^N$-periodic boundary values.
\begin{definition}
We say that a sequence $u_{\varepsilon}\in L^2(\Omega\times[0,T]\times D)$ two-scale converges to $u\in L^2(\Omega\times[0,T]\times D\times\mathbb{T}^N)$, if for every $\Psi\in L^2(\Omega\times[0,T]\times D;C_{\#}(\mathbb{T}^N))$ we have
\begin{equation}\nonumber
\begin{split}
\lim_{\varepsilon\to 0}\int_{\Omega}\int^T_0\int_Du_{\varepsilon}(\omega, t, x)&\Psi(\omega, t, x, \frac{x}{\varepsilon})dxdtd\mathbb{P}=\int_{\Omega}\int^T_0\int_D\int_{\mathbb{T}^N}u(\omega, t, x, y)\Psi(\omega, t, x, y)dydxdtd\mathbb{P}.\\
\end{split}
\end{equation}
\end{definition}

The following propositions are of great importance in obtaining the homogenization result.
\begin{proposition}
Assume that the sequence $u_{\varepsilon}$ is uniformly bounded in $L^2(\Omega\times[0,T]\times D)$. Then exists a subsequence, still denoted by $u_{\varepsilon}$, and a limit $u_0(\omega, t, x, y)\in L^2(\Omega\times[0,T]\times D\times\mathbb{T}^N)$ such that $u_{\varepsilon}$ two-scale converges to $u_0$.
\end{proposition}

\begin{proposition}
Assume that the sequence $u_{\varepsilon}$ is uniformly bounded in $L^2(\Omega\times[0,T]\times D)$, and the sequence $\varepsilon\nabla u_{\varepsilon}$ is also uniformly bounded in $L^2(\Omega\times[0,T]\times D)^N$. Then there exists a subsequence, still denoted by $u_{\varepsilon}$, and a limit $u_0(\omega, t, x, y)\in L^2(\Omega\times[0,T]\times D;H^1(\mathbb{T}^N))$ such that $u_{\varepsilon}$ two-scale converges to $u_0$ and $\varepsilon\nabla u_{\varepsilon}$ two-scale converges to $\nabla_yu_0$.
\end{proposition}

Notation.  For every function $\phi(x, y)$ defined on $D\times \mathbb{T}^N$, we denote by $\phi^{\varepsilon}$ the function $\phi(x, \frac{x}{\varepsilon})$.

\section{Effective Approximation}

We first present a priori estimate for the solution of the  stochastic parabolic equation (\ref{Equ.1}). For the following lemma, see \cite{Hy99}.
\begin{lemma}
Assume (H.1., H.2). For every $\varepsilon>0$, $u^0_{\varepsilon}\in H^1(D)$, and $T>0$, there exists a unique solution $u_{\varepsilon}\in L^2(\Omega;C([0,T]);L^2(D))\bigcap L^2(\Omega\times[0,T];H^1(D))$ of stochastic parabolic equation (\ref{Equ.1}) in the following sense:
\begin{equation}
(u_{\varepsilon}(t),v(t))=(u^0_{\varepsilon},v(0))-\int_0^t(\sigma^{\varepsilon}\nabla u_{\varepsilon}, \nabla v)ds+\int_0^t\left((\varepsilon^{-2}c^{\varepsilon}+d^{\varepsilon})u_{\varepsilon},v\right)ds
+\int_0^t(g^{\varepsilon}u_{\varepsilon},v)dW(s),
\end{equation}
for a.e. $\omega\in\Omega$, all $t\in[0,T]$ and for all $v\in C_c^{\infty}$. Moreover, there exists a constant $C_T$ that depends on $T$ such that
\begin{equation}
\mathbb{E}(\sup_{0\leqslant t\leqslant T}\|u_{\varepsilon}\|^2+\varepsilon^2\int_0^T\|u_{\varepsilon}\|^2_{H^1}dt)\leqslant C_T.
\end{equation}
\end{lemma}

\begin{theorem}
Assume (H.1.-H.3.) and that the initial data $u_{\varepsilon}^0\in H^1(D)$ is of the form
\begin{equation}
u_{\varepsilon}^0(x)=\psi_n(\frac{x}{\varepsilon},\theta^n)e^{2i\pi\frac{\theta^n\cdot x}{\varepsilon}}v^0(x),
\end{equation}
with $v^0\in H^1(D)$. Then the solution of (\ref{Equ.1}) can be expressed as
\begin{equation}
u_{\varepsilon}(t,x)=e^{i\frac{\lambda_n(\theta^n)t}{\varepsilon^2}}e^{2i\pi\frac{\theta^n\cdot x}{\varepsilon}}v_{\varepsilon}(t, x),
\end{equation}
where $v_{\varepsilon}$ two-scale converges  to $\psi_n(y,\theta^n)v(t, x)$, uniformly on compact time intervals in $\mathbb{R}^+$, and $v$ is the unique
solution of the effective stochastic parabolic equation
\begin{equation}\label{Homo.1}
\begin{cases}
\frac{\partial v}{\partial t} -{\rm div}(\sigma_n^*\nabla v)+d_n^*(x)v+g_n^*v\frac{dW(t)}{dt}=0  \quad {\rm in} \ D\times [0,T], \\
v=0 \quad {\rm on} \ \partial D\times [0,T], \\
v(0,x)=v^0(x)  \quad {\rm in} \ D.
\end{cases}
\end{equation}
Here $\sigma_n^*=\frac{1}{8\pi^2}\nabla_{\theta}\nabla_{\theta}\lambda_n(\theta^n)$, $d^*_n(x)=\int_{\mathbb{T}^N}d(x, y)|\psi_n(y)|^2dy$, and $g_n^*=\int_{\mathbb{T}^N}g(y)|\psi_n(y)|^2dy$.
\end{theorem}

\begin{proof}
Define a sequence $v_{\varepsilon}$ by
\begin{equation}\nonumber
v_{\varepsilon}(t,x)=u_{\varepsilon}(t, x)e^{-i\frac{\lambda_n(\theta^n)t}{\varepsilon^2}}e^{-2i\pi\frac{\theta^n\cdot x}{\varepsilon}}.
\end{equation}
Since $|v_{\varepsilon}|=|u_{\varepsilon}|$, by Lemma 1, we have
\begin{equation}
\mathbb{E}(\sup_{0\leqslant t\leqslant T}\|v_{\varepsilon}\|^2+\varepsilon^2\int_0^T\|v_{\varepsilon}\|^2_{H^1}dt)\leqslant C_T.
\end{equation}
By applying Proposition 2, up to a subsequence, there exists a limit $v^*(\omega, t, x, y)\in L^2(\Omega\times[0, T]\times D; H^1(\mathbb{T}^N))$ such that $v_{\varepsilon}$ two-scale converges to $v^*$ and $\varepsilon\nabla v_{\varepsilon}$ two-scale converges to $\nabla_yv^*$. Similarly, by definition of the initial data and Proposition 1, $v_{\varepsilon}(0, x)$ two-scale converges to $\psi_n(y, \theta^n)v^0(x)$.\\
{\it Step 1.}\ We multiply (\ref{Equ.1}) by the complex conjugate of $\varepsilon^2\phi(\omega, t, x, \frac{x}{\varepsilon}) e^{i\frac{\lambda_n(\theta^n)t}{\varepsilon^2}}e^{2i\pi\frac{\theta^n\cdot x}{\varepsilon}}$, where $\phi(\omega, t, x, \frac{x}{\varepsilon})$ is a smooth test function defined on $\Omega\times[0, T]\times D\times\mathbb{T}^N$, with compact support in $[0, T]\times D$, and with values in $\mathbb{C}$. Integrate with respect to $\omega\in\Omega$ and $t\in[0,T]$, and get
\begin{equation}
\begin{split}
&\varepsilon^2\int_{\Omega}\int_Du^0_{\varepsilon}\bar{\phi}^{\varepsilon}e^{-2i\pi\frac{\theta^n\cdot x}{\varepsilon}}dxd\mathbb{P}-\varepsilon^2\int_{\Omega}\int_0^T\int_Dv_{\varepsilon}\frac{\partial\bar{\phi}^{\varepsilon}}{\partial t}dxdtd\mathbb{P}\\
&+\int_{\Omega}\int_0^T\int_D\sigma^{\varepsilon}(\varepsilon\nabla+2i\pi\theta^n)v_{\varepsilon}\cdot(\varepsilon\nabla-2i\pi\theta^n)\bar{\phi}^{\varepsilon}dxdtd\mathbb{P}\\
&+\int_{\Omega}\int_0^T\int_D(c^{\varepsilon}-\lambda_n(\theta^n)+\varepsilon^2d^{\varepsilon})v_{\varepsilon}\bar{\phi}^{\varepsilon}dxdtd\mathbb{P}+\int_{\Omega}\int_0^T\int_D\varepsilon^2g^{\varepsilon}v^{\varepsilon}\bar{\phi}^{\varepsilon}dxdW(t)d\mathbb{P}=0.\\
\end{split}
\end{equation}
Passing to the two-scale limit term by term, we obtain
\begin{equation}
-(\text{div}_y+2i\pi\theta^n)\left(\sigma(y)(\nabla_y+2i\pi\theta^n)v^*\right)+c(y)v^*=\lambda_n(\theta^n)v^* \quad \text{in}\ \mathbb{T}^N,
\end{equation}
for a.e. $\omega\in\Omega$. By the simplicity of $\lambda_n(\theta^n)$ of Hypothesis H.3., we know that there exists a scalar function $v(\omega, t, x)\in L^2(\Omega\times[0, T]\times D)$ such that
\begin{equation}
v^*(\omega, t, x, y)=v(\omega, t, x)\psi_n(y, \theta^n).\\
\end{equation}
{\it Step 2.}\ We multiply (\ref{Equ.1}) by the complex conjugate of
 \begin{equation}\label{PsiConj.}
\Psi_{\varepsilon}=e^{i\frac{\lambda_n(\theta^n)t}{\varepsilon^2}}e^{2i\pi\frac{\theta^n\cdot x}{\varepsilon}}\left(\psi_n(\frac{x}{\varepsilon}, \theta^n)\phi(\omega, t, x)+\varepsilon\sum_{k=1}^N\frac{\partial \phi}{\partial x_k}(\omega, t, x)\zeta_k(\frac{x}{\varepsilon})\right),
\end{equation}
where $\phi(\omega, t, x)$ is a smooth compactly supported test function defined from $[0, T]\times D$ to $\mathbb{C}$, and $\zeta_k(y)$ is the solution of (\ref{SolutionZeta}). Then we obtain
\begin{equation}\label{Variation1}
\begin{split}
&\int_{\Omega}\int_Du^0_{\varepsilon}\bar{\Psi}_{\varepsilon}(t=0)dxd\mathbb{P}-\int_{\Omega}\int_0^T\int_Du_{\varepsilon}\frac{\partial\bar{\Psi}}{\partial t}dxdtd\mathbb{P}+\int_{\Omega}\int_0^T\int_D\sigma^{\varepsilon}\nabla u_{\varepsilon} \cdot\nabla\bar{\Psi}_{\varepsilon}dxdtd\mathbb{P}\\
&+\frac{1}{\varepsilon^2}\int_{\Omega}\int_0^T\int_Dc^{\varepsilon}u_{\varepsilon}\bar{\Psi}_{\varepsilon}dxdtd\mathbb{P}+\int_{\Omega}\int_0^T\int_Dd^{\varepsilon}u_{\varepsilon}\bar{\Psi}_{\varepsilon}dxdtd\mathbb{P}+\int_{\Omega}\int_0^T\int_Dg^{\varepsilon}u_{\varepsilon}\bar{\Psi}_{\varepsilon}dxdW(t)d\mathbb{P} =0.\\
\end{split}
\end{equation}
The summation convention that repeated indices indicate summation from 1 to $N$ is followed here. According to (\ref{PsiConj.}), we obtain
\begin{equation}\label{Variation2}
\begin{split}
&\int_{\Omega}\int_Du^0_{\varepsilon}\bar{\Psi}_{\varepsilon}(t=0)dxd\mathbb{P}-\int_{\Omega}\int_0^T\int_Dv_{\varepsilon}(\bar{\psi}_n^{\varepsilon}\frac{\partial\bar{\phi}}{\partial t}+\varepsilon\frac{\partial^2\bar{\phi}}{\partial x_k\partial t}\bar{\zeta_k}^{\varepsilon})dxdtd\mathbb{P}\\
&+\int_{\Omega}\int_0^T\int_D\sigma^{\varepsilon}\nabla u_{\varepsilon} \cdot\nabla\bar{\Psi}_{\varepsilon}dxdtd\mathbb{P}+\frac{1}{\varepsilon^2}\int_{\Omega}\int_0^T\int_D(c^{\varepsilon}-\lambda_n(\theta^n))v_{\varepsilon}\bar{\psi}_n^{\varepsilon}\bar{\phi}dxdtd\mathbb{P}\\
&+\frac{1}{\varepsilon}\int_{\Omega}\int_0^T\int_D(c^{\varepsilon}-\lambda_n(\theta^n))v_{\varepsilon}\frac{\partial\bar{\phi}}{\partial x_k}\bar{\zeta_k}^{\varepsilon}dxdtd\mathbb{P}+\int_{\Omega}\int_0^T\int_Dd^{\varepsilon}v_{\varepsilon}(\bar{\psi}_n^{\varepsilon}\bar{\phi}+\varepsilon\frac{\partial\bar{\phi}}{\partial x_k}\bar{\zeta_k}^{\varepsilon})dxdtd\mathbb{P}\\
&+\int_{\Omega}\int_0^T\int_Dg^{\varepsilon}v_{\varepsilon}(\bar{\psi}_n^{\varepsilon}\bar{\phi}+\varepsilon\frac{\partial\bar{\phi}}{\partial x_k}\bar{\zeta_k}^{\varepsilon})dxdW(t)d\mathbb{P} =0.\\
\end{split}
\end{equation}
After some algebra we find that
\begin{equation}\label{Div}
\begin{split}
&\int_D\sigma^{\varepsilon}\nabla u_{\varepsilon} \cdot\nabla\bar{\Psi}_{\varepsilon}dx=\int_D\sigma^{\varepsilon}(\nabla+2i\pi\frac{\theta^n}{\varepsilon})(\bar{\phi} v_{\varepsilon})\cdot(\nabla-2i\pi\frac{\theta^n}{\varepsilon})\bar{\psi}^{\varepsilon}_ndx\\
&+\varepsilon\int_D\sigma^{\varepsilon} (\nabla+2i\pi\frac{\theta^n}{\varepsilon})(\frac{\partial\bar{\phi}}{\partial x_k} v_{\varepsilon})\cdot(\nabla-2i\pi\frac{\theta^n}{\varepsilon})\bar{\zeta_k}^{\varepsilon}dx-\int_D\sigma^{\varepsilon}e_k\frac{\partial\bar{\phi}}{\partial x_k} v_{\varepsilon}\cdot(\nabla-2i\pi\frac{\theta^n}{\varepsilon})\bar{\psi}^{\varepsilon}_ndx\\
&+\int_D\sigma^{\varepsilon}(\nabla+2i\pi\frac{\theta^n}{\varepsilon})(\frac{\partial\bar{\phi}}{\partial x_k} v_{\varepsilon})\cdot e_k\bar{\psi}^{\varepsilon}_ndx-\int_D\sigma^{\varepsilon}v_{\varepsilon}\nabla\frac{\partial\bar{\phi}}{\partial x_k}\cdot e_k\bar{\psi}^{\varepsilon}_ndx\\
&-\int_D\sigma^{\varepsilon}v_{\varepsilon}\nabla\frac{\partial\bar{\phi}}{\partial x_k}\cdot(\varepsilon\nabla-2i\pi\theta^n)\bar{\zeta_k}^{\varepsilon}dx+\int_D\sigma^{\varepsilon}\bar{\zeta_k}^{\varepsilon}(\varepsilon\nabla+2i\pi\theta^n)v_{\varepsilon}\cdot\nabla\frac{\partial\bar{\phi}}{\partial x_k}dx.\\
\end{split}
\end{equation}
For every smooth compactly supported test function $\Phi$, we deduce from the definition of $\psi_n$ that
\begin{equation}\label{PsiDef.}
\int_D\sigma^{\varepsilon}(\nabla+2i\pi\frac{\theta^n}{\varepsilon})\psi_n^{\varepsilon}\cdot(\nabla-2i\pi\frac{\theta^n}{\varepsilon})\bar{\Phi}dx+\frac{1}{\varepsilon^2}\int_D(c^{\varepsilon}-\lambda_n(\theta^n))\psi_n^{\varepsilon}\bar{\Phi}dx=0,
\end{equation}
and from the definition of $\zeta_k$,
\begin{equation}\label{ZetaDef.}
\begin{split}
&\int_D\sigma^{\varepsilon}(\nabla+2i\pi\frac{\theta^n}{\varepsilon})\zeta_k^{\varepsilon}\cdot(\nabla-2i\pi\frac{\theta^n}{\varepsilon})\bar{\Phi}dx+\frac{1}{\varepsilon^2}\int_D(c^{\varepsilon}-\lambda_n(\theta^n))\zeta_k^{\varepsilon}\bar{\Phi}dx=\\
&\varepsilon^{-1}\int_D\sigma^{\varepsilon}(\nabla+2i\pi\frac{\theta^n}{\varepsilon})\psi_n^{\varepsilon}\cdot e_k\bar{\Phi}dx-\varepsilon^{-1}\int_D\sigma^{\varepsilon}e_k\psi_n^{\varepsilon}\cdot(\nabla-2i\pi\frac{\theta^n}{\varepsilon})\bar{\Phi}dx.\\
\end{split}
\end{equation}

Combining (\ref{Div}) with the other terms of (\ref{Variation2}), we find that the first term of its right-hand side cancels out because of (\ref{PsiDef.}) with $\Phi=\bar{\phi}v_{\varepsilon}$, and the next three terms cancel out because of (\ref{ZetaDef.}) with $\Phi=\frac{\partial\bar{\phi}}{\partial x_k}v_{\varepsilon}$.
Finally, (\ref{Equ.1}) multiplied by $\bar{\Psi}_{\varepsilon}$ yields 
\begin{equation}\label{Variation3}
\begin{split}
&\int_{\Omega}\int_Du^0_{\varepsilon}\bar{\Psi}_{\varepsilon}(t=0)dxd\mathbb{P}-\int_{\Omega}\int_0^T\int_Dv_{\varepsilon}(\bar{\psi}_n^{\varepsilon}\frac{\partial\bar{\phi}}{\partial t}+\varepsilon\frac{\partial^2\bar{\phi}}{\partial x_k\partial t}\bar{\zeta_k}^{\varepsilon})dxdtd\mathbb{P}\\
&-\int_{\Omega}\int_0^T\int_D\sigma^{\varepsilon}v_{\varepsilon}\nabla\frac{\partial\bar{\phi}}{\partial x_k}\cdot e_k\bar{\psi}^{\varepsilon}_ndxdtd\mathbb{P}-\int_{\Omega}\int_0^T\int_D\sigma^{\varepsilon}v_{\varepsilon}\nabla\frac{\partial\bar{\phi}}{\partial x_k}\cdot(\varepsilon\nabla-2i\pi\theta^n)\bar{\zeta_k}^{\varepsilon}dxdtd\mathbb{P}\\
&+\int_{\Omega}\int_0^T\int_D\sigma^{\varepsilon}\bar{\zeta_k}^{\varepsilon}(\varepsilon\nabla+2i\pi\theta^n)v_{\varepsilon}\cdot\nabla\frac{\partial\bar{\phi}}{\partial x_k}dxdtd\mathbb{P}+\int_{\Omega}\int_0^T\int_Dd^{\varepsilon}v_{\varepsilon}(\bar{\psi}_n^{\varepsilon}\bar{\phi}+\varepsilon\frac{\partial\bar{\phi}}{\partial x_k}\bar{\zeta_k}^{\varepsilon})dxdtd\mathbb{P}\\
&+\int_{\Omega}\int_0^T\int_Dg^{\varepsilon}v_{\varepsilon}(\bar{\psi}_n^{\varepsilon}\bar{\phi}+\varepsilon\frac{\partial\bar{\phi}}{\partial x_k}\bar{\zeta_k}^{\varepsilon})dxdW(t)d\mathbb{P} =0.\\
\end{split}
\end{equation}
Passing to the two-scale limit in each term of (\ref{Variation3}) gives
\begin{equation}\nonumber
\begin{split}
&\int_{\Omega}\int_D\int_{\mathbb{T}^N}\psi_nv^0\bar{\psi}_n\bar{\phi}(t=0)dydxd\mathbb{P}-\int_{\Omega}\int_0^T\int_D\int_{\mathbb{T}^N}\psi_nv\bar{\psi}_n\frac{\partial\bar{\phi}}{\partial t}dydxdtd\mathbb{P}\\
&-\int_{\Omega}\int_0^T\int_D\int_{\mathbb{T}^N}\sigma\psi_nv\nabla\frac{\partial\bar{\phi}}{\partial x_k}\cdot e_k\bar{\psi}_ndydxdtd\mathbb{P}-\int_{\Omega}\int_0^T\int_D\int_{\mathbb{T}^N}\sigma\psi_nv\nabla\frac{\partial\bar{\phi}}{\partial x_k}\cdot(\nabla_y-2i\pi\theta^n)\bar{\zeta_k}dydxdtd\mathbb{P}\\
&+\int_{\Omega}\int_0^T\int_D\int_{\mathbb{T}^N}\sigma\bar{\zeta_k}(\nabla_y+2i\pi\theta^n)\psi_nv\cdot\nabla\frac{\partial\bar{\phi}}{\partial x_k}dydxdtd\mathbb{P}+\int_{\Omega}\int_0^T\int_D\int_{\mathbb{T}^N}d(x, y)\psi_nv\bar{\psi}_n\bar{\phi}dydxdtd\mathbb{P}\\
&+\int_{\Omega}\int_0^T\int_D\int_{\mathbb{T}^N}g(y)\psi_nv\bar{\psi}_n\bar{\phi}dydxdW(t)d\mathbb{P} =0.\\
\end{split}
\end{equation}
Recalling the normalization $\int_{\mathbb{T}^N}|\psi_n|^2dy=1$, and introducing
\begin{equation}\label{Sigma}
\begin{split}
(\sigma_n^*)_{jk}&=\frac{1}{2}\int_{\mathbb{T}^N}(\sigma\psi_ne_j\cdot e_k\bar{\psi}_n+\sigma\psi_ne_k\cdot e_j\bar{\psi}_n+\sigma\psi_ne_j\cdot(\nabla_y-2i\pi\theta^n)\bar{\zeta_k}\\
&+\sigma\psi_ne_k\cdot(\nabla_y-2i\pi\theta^n)\bar{\zeta_j}-\sigma\bar{\zeta_k}(\nabla_y+2i\pi\theta^n)\psi_n\cdot e_j-\sigma\bar{\zeta_j}(\nabla_y+2i\pi\theta^n)\psi_n\cdot e_k)dy,
\end{split}
\end{equation}
and $d^*_n(x)=\int_{\mathbb{T}^N}d(x, y)|\psi_n(y)|^2dy$, $g_n^*=\int_{\mathbb{T}^N}g(y)|\psi_n(y)|^2dy$, the limit equation is equivalent to
\begin{equation}
\begin{split}
&\int_{\Omega}\int_Dv^0\bar{\phi}dxd\mathbb{P}-\int_{\Omega}\int_0^T\int_Dv\frac{\partial\bar{\phi}}{\partial t}dxdtd\mathbb{P}-\int_{\Omega}\int_0^T\int_D\sigma_n^*v\cdot\nabla\nabla\bar{\phi}dxdtd\mathbb{P}\\
&+\int_{\Omega}\int_0^T\int_Dd^*(x)v\bar{\phi}dxdtd\mathbb{P}+\int_{\Omega}\int_0^T\int_Dg_n^*v\bar{\phi}dxdW(t)d\mathbb{P}=0,\\
\end{split}
\end{equation}
which is a weak form of the homogenized equation (\ref{Homo.1}). The compatibility condition of (\ref{SolutionChi}) for the second derivative of $\psi_n$ shows that $\sigma_n^*$, defined by (\ref{Sigma}), is indeed equal to $\frac{1}{8\pi^2}\nabla_{\theta}\nabla_{\theta}\lambda_n(\theta^n)$.

This completes the proof.
\end{proof}

\section*{Acknowledgment}
 This work was partly supported by the NSFC grants 11531006 and 11771449.

\end{document}